\documentclass[12pt]{article}

\usepackage{amsmath,amsfonts,amssymb}
\usepackage{fullpage} 
\usepackage[colorlinks]{hyperref} 
\usepackage[parfill]{parskip} 
\usepackage{enumitem} 
\usepackage{graphicx}
\usepackage{amsthm}

\begingroup
    \makeatletter
    \@for\theoremstyle:=definition,remark,plain\do{%
        \expandafter\g@addto@macro\csname th@\theoremstyle\endcsname{%
            \addtolength\thm@preskip\parskip
            }%
        }
\endgroup

\newtheorem{theorem}{Theorem}[section]

\newtheorem{lemma}[theorem]{Lemma}

\newtheorem{claim}[theorem]{Claim}

\newcommand{\naturals}{\ensuremath{\mathbb{N}}}

\title{Ordered Size Ramsey Number of Paths}

\author{%
  J\'ozsef Balogh \footnote{Department of Mathematical Sciences, University of Illinois at Urbana-Champaign, Urbana, Illinois 61801, USA, and Moscow Institute of Physics and Technology, 9 Institutskiy per., Dolgoprodny, Moscow Region,141701, Russian Federation. Email: jobal@math.uiuc.edu. Research is partially supported by NSF Grant DMS-1500121, Arnold O. Beckman Research
Award (UIUC Campus Research Board RB 18132) and the Langan Scholar Fund (UIUC).}%
 \and Felix Christian Clemen \footnote {Department of Mathematics, University of Illinois at Urbana-Champaign, Urbana, Illinois 61801, USA, Email: fclemen2@illinois.edu.}%
 \and Emily Heath \footnote {Department of Mathematics, University of Illinois at Urbana-Champaign, Urbana, Illinois 61801, USA, Email: eheath3@illinois.edu}%
 \and Mikhail Lavrov \footnote {Department of Mathematics, University of Illinois at Urbana-Champaign, Urbana, Illinois 61801, USA, Email: mlavrov@illinois.edu}%
  }

\begin{document}
\maketitle

\begin{abstract}
An \emph{ordered graph} is a simple graph with an ordering on its vertices.
Define the \emph{ordered path} $P_n$ to be the monotone increasing path with $n$ edges.
The \emph{ordered size Ramsey number} $\tilde{r}(P_r,P_s)$ is the minimum number $m$ for which there exists an ordered graph $H$ with $m$ edges such that every two-coloring of the edges of $H$ contains a red copy of $P_r$ or a blue copy of $P_s$.
For $2\leq r\leq s$, we show $\frac{1}{8}r^2s\leq \tilde{r}(P_r,P_s)\leq Cr^2s(\log s)^3$, where $C>0$ is an absolute constant. This problem is motivated by the recent results of Buci\'c--Letzter--Sudakov~\cite{BLS} and Letzter--Sudakov~\cite{LS} for oriented graphs.
\end{abstract}

\section{Introduction} 

Ramsey theory is the branch of combinatorics which studies the forced appearance of certain substructures in sufficiently large structures. 
One of the classical results in the field is Ramsey's theorem \cite{R}, which states that for any graph $G$, every sufficiently large two-edge-colored complete graph must contain a monochromatic copy of $G$. 
It is natural to ask for the minimum order of a complete graph with this property, which we refer to as the \emph{Ramsey number} of $G$, denoted $r(G)$.

The study of Ramsey numbers has branched out in many different directions.
For example, Erd\H{o}s, Faudree, Rousseau, and Schelp \cite{EFRS} introduced the idea of the size Ramsey number in 1972. The \emph{size Ramsey number} of a graph $G$, denoted $\hat{r}(G)$, is the minimum integer $m$ for which there is a graph $H$ with $m$ edges such that every two-coloring of $E(H)$ contains a monochromatic copy of $G$.

Let $P_n$ denote the path on $n$ vertices.
It is well-known that $r(P_n)=\Theta(n)$.
Since the complete graph on $r(G)$ vertices has $\binom{r(G)}{2}$ edges, this implies $\hat{r}(P_n)=O(n^2)$. Erd\H{o}s~\cite{E} offered a \$100 reward for determining the asymptotic behavior of $\hat{r}(P_n)$; the question was settled in 1983 by Beck~\cite{B}, who showed that that $\hat{r}(P_n)$ is also linear in $n$.
The constant in this upper bound has been improved several times (as in \cite{Bo}, \cite{DP}, \cite{L}),  and the current best upper bound, given by Dudek and Pra\l at \cite{DP2}, is $\hat{r}(P_n)\leq 74n$.

Beck \cite{B2} also gave the first nontrivial lower bound on $\hat{r}(P_n)$.
This result was later improved by Bollob\'as \cite{Bo2}, who showed that $\hat{r}(P_n)\geq (1+\sqrt{2})n-O(1)$, and then by Dudek and Pra\l at~\cite{DP2}, who gave the current best lower bound of $\hat{r}(P_n)\geq 5n/2-O(1)$.

More generally, Krivelevich \cite{K} showed that  $\hat{r}(P_n,k)$, the size Ramsey number of the path with respect to edge-coloring with $k$ colors, satisfies $\hat{r}(P_n,k)=O((\log k)k^2n)$.
Dudek and Pra\l at \cite{DP3} have also provided an alternative proof of this upper bound, which is nearly optimal since $\hat{r}(P_n,k)=\Omega(k^2n)$.

The question of determining the size Ramsey number of the path is also of interest in the setting of oriented graphs.
An \emph{oriented graph} is a directed graph in which at most one of $xy$ and $yx$ appears as an edge for each pair of vertices $x$ and $y$.
The \emph{oriented size Ramsey number} of an oriented graph $G$, denoted $\vec{r}(G)$, is the minimum number $m$ for which every two-coloring of the edges of some oriented graph $H$ with $m$ edges contains a monochromatic copy of $G$.
One significant difference arises between the undirected and oriented cases: the same undirected graph admits many different orientations, and these orientations may have very different oriented size Ramsey numbers.
We restrict our attention to the monotone increasing path, which we refer to as the \emph{ordered path}.
Throughout the rest of the paper, we denote the ordered path on $n+1$ vertices and $n$ edges simply by $P_n$.

While the size Ramsey number of an undirected path is linear, the oriented size Ramsey number of the ordered path, $\vec{r}(P_n)$, is very different.
Following initial results by Ben-Eliezer, Krivelevich, and Sudakov \cite{BKS} bounding $\vec{r}(P_n)$, Buci\'c, Letzter, and Sudakov \cite{BLS} recently gave a nice proof showing that $\vec{r}(P_n)= O(n^2\log n)$ by giving a lower bound on the longest monochromatic path in two-coloured random tournaments.
Letzter and Sudakov \cite{LS} also gave a matching lower bound which shows that $\vec{r}(P_n)=\Theta(n^2\log n)$.

In this paper, we consider ordered graphs.
An \emph{ordered graph} on $n$ vertices is a simple graph whose vertices have been labeled with $\{1, 2, \ldots, n\}$.
Since ordered graphs can be viewed as acylic oriented graphs, any lower bound on the oriented size Ramsey number of $P_n$ also applies when we restrict our attention to ordered graphs.
However, one significant difference between the study of ordered and oriented graphs is the lack of symmetry in the ordered case.
For example, the edge 12 in an $n$-vertex ordered graph $G$ plays a very different role in ordered subgraphs of $G$ than the edge $1n$. 

The Erd\H{o}s--Szekeres Theorem \cite{ESz} states that every sequence of $n^2+1$ distinct real numbers must contain an increasing or decreasing subsequence of $n+1$ numbers. 
This result can be interpreted as giving the least number of vertices in an ordered graph such that any two-coloring of the edges contains a monochromatic ordered path with $n+1$ vertices.
Here, we consider the \emph{ordered size Ramsey number}, minimizing the number of edges in ordered graphs with this property rather than the number of vertices.
Formally, let $\tilde r(P_r, P_s)$ denote the minimum number of edges in an ordered graph for which any red-blue coloring of the edges contains either a red ordered $P_r$ or a blue ordered $P_s$.

Our main result is the following theorem:

\begin{theorem}
	\label{thm:size-ramsey}
	For some absolute constant $C>0$ and for all $2 \le r \le s$,
	\[
		\frac18 r^2s\le \tilde r(P_r, P_s) \le C r^2 s (\log s)^3.
	\]
\end{theorem}

It would be interesting to determine the true asymptotic behavior of $\hat{r}(P_r,P_s)$, although we do not have a conjecture.

In Section 2, we prove the upper bound of Theorem~\ref{thm:size-ramsey}.
While random and pseudorandom graphs have been used to prove upper bounds on several variants of size Ramsey numbers (for example, see Beck~\cite{B} and Alon--Chung~\cite{AC}), our proof is the first which uses inhomogeneous random graphs to obtain a result of this type.

In Section 3, we prove the lower bound of Theorem~\ref{thm:size-ramsey}.
Our strategy is to adapt an edge-coloring algorithm used by Reimer~\cite{Re} to show that the directed size Ramsey number of $P_n$ is $\Omega(n^2)$.

In Section 4, we describe an alternative approach used to obtain an upper bound on $\tilde r(P_r, P_s)$ which transforms the question of edge-coloring into a problem about vertex-coloring. 
In order to give an upper bound on the $b$-color ordered Ramsey number of the ordered path $P_s$, we iteratively define a sequence of graphs, the last of which proves the following theorem:

\begin{theorem} 
	\label{thm:vertex-coloring}
	There exists an ordered graph with $se^{O(\log b \sqrt{\log s})}$ edges for which every $b$-coloring of the vertices contains a monochromatic ordered path of length~$s$.
\end{theorem}

As mentioned at the end of Section~\ref{section:vertex-coloring}, a bound of this form can be used to obtain an upper bound on $\tilde r(P_r,P_s)$.  
Theorem~\ref{thm:vertex-coloring} is not sufficiently strong to improve on the bound of Theorem~\ref{thm:size-ramsey} in this way, but a stronger bound of this type would allow us to deduce a better upper bound on the ordered size Ramsey number of ordered paths.

Finally, in Section 5, we consider $ \tilde r(P_n;q)$, the minimum number of edges in an ordered graph for which any $q$-coloring of the edges contains a monochromatic ordered $P_n$. Using proofs similar to those in Sections 2 and 3, we obtain upper and lower bounds on $r(P_n;q)$ which agree up to a polylogarithmic factor. 

\begin{theorem}
	\label{thm:multicolor}
	There exists an absolute constant $C > 0$ such that for all $n \ge 2$ and $q \ge 2$,
	\[
		\frac{n^{2q-1}}{2^{2q-1}(q-1)!} \le \tilde r(P_n;q) \le C n^{2q-1} (\log n)^3.
	\]
\end{theorem}

Throughout this paper, all logarithms are assumed to be natural, and we omit floor and ceiling signs in our proofs for convenience. 
For a graph $G$ and $S\subseteq V(G)$, we write $G[S]$ for the subgraph of $G$ induced by $S$ and $G-S$ for the subgraph of $G$ induced by the complement of $S$.
For vertices $u,v$ in an ordered graph $G$, let $d(u,v)$ denote the distance between $u$ and $v$ in the ordering of the vertices of $G$.

\section{Proof of the Upper Bound of Theorem~\ref{thm:size-ramsey}}

To prove the upper bound of Theorem~\ref{thm:size-ramsey}, we show the existence of an ordered graph on $n = 4rs$ vertices and $O(r^2 s (\log s)^3)$ edges for which any red-blue coloring of the edges contains either a red ordered $P_r$ or a blue ordered $P_s$. 

We assume throughout the proof that $s=2^t$ for some $t\in\naturals$; this does not change the asymptotic form of the upper bound. Additionally, we may assume that $s$ is sufficiently large, by choosing the constant $C$ in Theorem~\ref{thm:size-ramsey} so that the upper bound holds for small~$s$.

For $i \ge 0$, define the parameters 
\[
	p_i = 2^{-i},\quad \ell_i = 2^{i+3}t,\quad \text{and}\quad m_i = 2^{i+7}r t^2.
\]
Let $k$ be the integer satisfying $\frac n2 < m_k \le n$. Since $m_{t} = 2^7rs t^2 > n$, we have $k < t$. However, for $s$ sufficiently large, we have $k \ge 0$, since $m_0 = 2^7 rt^2 < 4rs$ for large $s$.

\begin{lemma}
	\label{lemma:random-graph}
	Assume $s$ is sufficiently large. For all $0 \le i \le k$, there exists an ordered graph $G_i$ on $n$ vertices with at most $2n m_i p_i$ edges such that between any two disjoint sets $S_1, S_2 \subseteq V(G_i)$ with $|S_1| = |S_2| = \ell_i$ and $\max\{d(u,v) : u \in S_1, v \in S_2\} \le m_i$, there is at least one edge.
\end{lemma}
\begin{proof}
When $i=0$, then $p_0 = 1$ and we may take $G_0$ to be the graph which contains an edge $uv$ whenever $d(u,v) \le m_0$.

When $1 \le i \le k$, we take $G_i$ to be the random ordered graph on $n$ vertices which contains an edge $uv$ with probability $p_i$ when $d(u,v) \le m_i$ and with probability $0$ otherwise. We show that with positive probability, such a graph satisfies both of the desired properties.

The expected number of edges in $G_i$ is less than $n m_i p_i$, so with probability at least $\frac12$, $G_i$ contains no more than $2nm_ip_i$ edges.

The expected number of ``bad" pairs of sets $(S_1,S_2)$ violating the conclusion of the lemma is bounded by
\begin{align*}
\binom{n}{\ell_i}^2 (1-p_i)^{\ell_i^2} 
	&\le \exp\left(2\ell_i \log (4rs) - p_i \ell_i^2\right) \le \exp\left(2^{i+4}t \log (4s^2) - 2^{i+6} t^2\right)\\
	&\le \exp\left(2^{i+4}t (2\log s + \log 4 - 4t)\right) \le \exp(-2^{i+4}t) \le s^{-16}
\end{align*}
for $s \ge 9$, where we have $2\log s + \log 4 - 4t < -1$.

Therefore, a randomly chosen $G_i$ contains a bad pair with probability at most $s^{-16}$. With probability at least $\frac12 - s^{-16} > 0$, $G_i$ satisfies both properties: it has fewer than $2nm_ip_i$ edges and no bad pairs $(S_1,S_2)$. In particular, some such ordered graph $G_i$ must exist.
\end{proof}

Let $G$ be the union of the ordered graphs $G_0, G_1, \dots, G_k$ constructed in Lemma~\ref{lemma:random-graph}. This graph has at most
\[
	\sum_{i=0}^k 2nm_i p_i = \sum_{i=0}^k 2(4rs)(2^{i+7}rt^2)(2^{-i}) = 2^{10}(k+1)r^2st^2 \le C r^2 s (\log s)^3
\]
edges, for some constant $C$. We will show that every red-blue coloring of the edges of $G$ contains a red $P_r$ or a blue $P_s$, proving the upper bound of Theorem~\ref{thm:size-ramsey}.

Fix an arbitrary red-blue coloring of the edges of $G$. For every $v \in V(G)$, denote by $R(v)$ the length of the longest red path ending on $v$. If there is a vertex $v$ with $R(v) \ge r$, then we are done, so we assume that $R(v)$ has values in $\{0, 1, \dots, r-1\}$.

By the pigeonhole principle, there must be some set $A \subseteq V(G)$ of size at least $\frac{n}{r} = 4s$ such that $R(u) = R(v)$ for all $u,v \in A$. From now on, we work only in the graph $G[A]$. Every edge of $G[A]$ must be blue, since a red edge $uv$ in $G[A]$ would imply $R(v) \ge R(u) + 1$. Therefore, it is sufficient to find a $P_s$ in $G[A]$ to prove the theorem.

We say that two paths are \emph{non-overlapping} if every vertex of one precedes every vertex of the other.

\begin{lemma}
	\label{lemma:induction}
	For $0 \le j \le k$, we can find a collection $\mathcal P_j$ of pairwise non-overlapping paths in $G[A]$ satisfying the following conditions:
	\begin{enumerate}[label=(\alph*)]
	\item \label{c:num-vertices} Together, the paths in $\mathcal P_j$ include at least $(4 - \frac jk)\cdot s$ vertices of $A$.
	\item \label{c:num-paths} The number of paths in $\mathcal P_j$ is at most $\frac{n}{m_j} + 1$.
	\end{enumerate}
\end{lemma}
\begin{proof}
We induct on $j$.

To construct $\mathcal P_0$, we choose the paths greedily, using only edges in $G_0[A]$. Start a path from the vertex $a_1$ and add vertices $a_2, a_3, \dots$, until we reach a vertex $a_i$ whose distance to $a_{i+1}$ exceeds $m_0$. When this happens, start another path from the vertex $a_{i+1}$ and proceed in the same way. 

After each path except possibly the last, there is a gap of at least $m_0$ vertices not present in $A$. Thus, there are at most $\frac{n}{m_0} + 1$ paths, and condition~\ref{c:num-paths} is satisfied. 
Since every vertex in $A$ is included in a path in $\mathcal P_0$, and $|A| \ge 4s$, condition~\ref{c:num-vertices} holds as well.

Next, for $1 \le j \le k$, we will use $\mathcal P_{j-1}$ to construct $\mathcal P_j$.

First, we remove all short paths from $\mathcal P_{j-1}$: paths that have at most $2\ell_j$ vertices. To simplify analysis for large $j$, if all paths have at most $2\ell_j$ vertices, we keep an arbitrary path.

Second, we combine paths that are close together into a single path. More precisely, define the \emph{gap} between paths $P$ and $Q$, with $P$ preceding $Q$, to be the distance between the $\ell_j$\textsuperscript{th}-to-last vertex of $P$ and the $\ell_j$\textsuperscript{th} vertex of $Q$. Whenever two consecutive paths $P$ and $Q$ have a gap between them which is shorter than $m_j$, we combine them into one path.

This is always possible by applying Lemma~\ref{lemma:random-graph}, which guarantees that there is an edge in $G_j$ between the last $\ell_j$ vertices of $P$ and the first $\ell_j$ vertices of $Q$. The resulting path using this edge might skip some of the last vertices of $P$ and some of the first vertices of $Q$, but we lose fewer than $\ell_j$ vertices from each: fewer than $2\ell_j$ vertices total. Together, $P$ and $Q$ have more than $4\ell_j$ vertices, so the combined path still has more than $2\ell_j$ vertices, which means we can continue combining paths until no more paths have a gap shorter than $m_j$.

After we are done, we verify that the resulting collection of paths $\mathcal P_j$ satisfies the conditions of the lemma. First, since we start with at most $\frac{n}{m_{j-1}} + 1$ paths and are left with at least one path, we perform at most $\frac{n}{m_{j-1}}$ steps of either deleting a short path or combining two paths. Each step discards at most $2\ell_j$ vertices, so we discard at most 
\[
	2\ell_j \cdot \frac{n}{m_{j-1}} = 2^{j+4} t \cdot \frac{4rs}{2^{j+6}rt^2} = \frac{s}{t} < \frac{s}{k}
\] vertices in total. Therefore the paths in $\mathcal P_j$ still include at least $(4 - \frac{j-1}{k}) \cdot s  - \frac sk = (4 - \frac jk) \cdot s$ vertices of $A$, and condition~\ref{c:num-vertices} holds.

Second, we know that all remaining gaps between paths are at least $m_j$ in length. Also, the gap before a path $P$ ends at its $\ell_j$\textsuperscript{th} vertex, and the gap after $P$ starts at its $\ell_j$\textsuperscript{th}-to-last vertex. Since $P$ has at least $2\ell_j$ vertices, each gap ends before the next gap starts, and hence the gaps represent disjoint intervals of at least $m_j$ vertices. Therefore, the number of gaps can be at most $\frac{n}{m_j}$, which means there are at most $\frac{n}{m_j}+1$ paths, verifying condition~\ref{c:num-paths}. This completes the induction step.
\end{proof}

After $k$ steps, we have at most $\frac{n}{m_k} + 1 < 3$ paths which together include at least $3s$ vertices of $A$. As a result, one of the paths has length at least $s$, completing the proof of the upper bound of Theorem~\ref{thm:size-ramsey}.

\section{Proof of the Lower Bound of Theorem~\ref{thm:size-ramsey}}

We prove the lower bound of Theorem~\ref{thm:size-ramsey} by extending an argument of Reimer \cite{Re}.
Our innovation is to use an Erd\H{o}s--Szekeres coloring on a dense subgraph of the given ordered graph $G$.

\begin{proof}[Proof of the Lower Bound of Theorem~\ref{thm:size-ramsey}]
Let $d < r$ be a parameter to be chosen later, and let $G$ be an arbitrary ordered graph with at most $d(r-d)(s-d)$ edges.

Define $U_0$ to be the set of vertices with at least $d$ neighbors preceding them in the order on $G$. Then $|U_0|\leq\frac1d|E(G)| = (r-d)(s-d)$. Apply an Erd\H{o}s--Szekeres coloring to $G[U_0]$: partition $U_0$ into $s-d$ consecutive blocks of size $r-d$, color edges within each block red, and color edges between different blocks blue. Then $G[U_0]$ does not contain a red path on more than $r-d$ vertices or a blue path on more than $s-d$ vertices.

The remaining graph $G - U_0$ has chromatic number at most $d$, by a greedy coloring that considers the vertices as they are ordered in $G$. Let $U_1, U_2, \dots, U_d$ be the color classes of a proper $d$-coloring of $G - U_0$.

To color the remaining edges of $G$, let $uv$ be an edge with $u \in U_i$, $v \in U_j$, and $u$ preceding $v$ in the order on $V(G)$. Color the edge $uv$ blue if $i<j$, otherwise red.

A red path may visit a single vertex in $U_d$, then a single vertex in $U_{d-1}, U_{d-2}$, and so on; finally, it may visit up to $r-d$ vertices in $U_0$. Therefore, the red path can have at most $r$ vertices. Similarly, a blue path can have at most $s$ vertices.

Recall that we assume $r \le s$. If we set $d = \frac12r$, then this argument gives an algorithm for coloring  the edges of any ordered graph on at most $\frac r2 \cdot \frac r2 \cdot (s - \frac r2) \ge \frac18 r^2s$ edges without creating a red $P_r$ or blue $P_s$.
\end{proof}

Note that the optimal choice of $d$ ranges from $\frac13r$ when $r=s$ to $\frac12r$ when $r \ll s$, but this only affects the constant.

\section{Proof of Theorem~\ref{thm:vertex-coloring}}
\label{section:vertex-coloring}

We prove Theorem~\ref{thm:vertex-coloring} by iteratively constructing a sequence of ordered graphs $G_1, \ldots, G_t$ where $t=\sqrt{\log s}$ in which the last graph $G_t$ has the desired property that every $b$-coloring of the vertices of $G_t$ contains a monochromatic ordered path of length $s$.

In order to describe this sequence, we must first define a bipartite graph which will appear in our construction.
The existence of this bipartite graph is a standard fact in random graph theory following from the first moment method.

{\lemma Let $n,m\in\naturals$ and $m\leq \frac{n}{2}$. There exists a bipartite graph with classes of size $n$, at most $O(\frac{n^2}{m}\log\frac{n}{m})$ edges, and the property that there is some edge between any pair of $m$-sets of the partite classes.}

For each $n$ and $m$, fix such a graph and denote it by $H(n,m)$. 
Now we are ready to prove Theorem~\ref{thm:vertex-coloring}.

\begin{proof}[Proof of Theorem~\ref{thm:vertex-coloring}]

Let $A=2be^{\sqrt{\log s}}$ and $k_\ell=\left(\frac{A}{10b}\right)^{\ell}$ for $1\leq \ell\leq t$.
Define a sequence of ordered graphs as follows:
let $G_1$ be the ordered complete graph $K_{A}$. 
For $\ell\geq 2$, let $G_{\ell}$  be the ordered graph formed by placing $A$ copies of $G_{\ell-1}$ consecutively and including a copy of $H(A^{\ell-1},k_{\ell-1})$ between any two copies of $G_{\ell-1}$.

We show in the following two claims that every $b$-coloring of the vertices of $G_{\ell}$ contains a relatively long monochromatic ordered path compared to the size of the graph.
Applying this argument to $G_t$ with $t=\sqrt{\log s}$ will give our result.

\begin{claim}
	\label{claim:long-path}
	Any $b$-coloring of the vertices of $G_{\ell}$ contains a monochromatic ordered path of length at least $A^{\ell}\left(\frac{1}{2b}\right)^{\ell}$.
	\end{claim}

\begin{proof}
We prove this claim by induction on $\ell$. 

First, note that any $b$-coloring of the vertices of $G_1=K_A$ must contain a monochromatic ordered path of length at least $\frac{A}{b}$.
Assume the claim holds for $G_{\ell-1}$ and consider an arbitrary $b$-coloring of $V(G_{\ell})$.

For an ordered graph $G$, let $L(G)$ denote the maximum length of a monochromatic ordered path which is guaranteed in any $b$-coloring of the vertices of $G$.
We say that a path in $G_{\ell-1}$ is \emph{long} if it has length at least $A^{\ell-1}\left(\frac{1}{2b}\right)^{\ell-1}$.

By the inductive hypothesis, each copy of $G_{\ell-1}$ contains a monochromatic long path in one of the $b$ colors.
Therefore, without loss of generality, at least $\frac{A}{b}$ of the $A$ copies of $G_{\ell-1}$ contain a long path of color 1.
Denote these copies of $G_{\ell-1}$ by $H_1, H_2, \ldots, H_j$, noting that $j\geq\frac{A}{b}$.

Since each pair of copies of $G_{\ell-1}$ is connected by the bipartite graph $H(A^{\ell-1},k_{\ell-1})$ in our construction of $G_{\ell}$, there must be an edge between one of the last $k_{\ell-1}$ vertices of the long color-1-path in $H_i$ and the first $k_{\ell-1}$ vertices of the long color-1-path in $H_{i+1}$ for each $1\leq i\leq j-1$. 
Therefore, by connecting the long color-1-paths in consecutive ``color 1" copies of $G_{\ell-1}$ with edges from the bipartite graphs in this way, we can find an ordered color-1-path in $G_{\ell}$ of length at least $\frac{A}{b}\left(L(G_{\ell-1})-2k_{\ell-1}\right)$. 
Applying our inductive hypothesis, this yields a monochromatic path in $G_{\ell}$ of length at least 
\begin{align*}
L(G_{\ell}) \geq \frac{A}{b}\left(A^{\ell-1}\left(\frac{1}{2b}\right)^{\ell-1}-2\left( \frac{A}{10b}\right)^{\ell-1}\right) = \frac{A^\ell}{b^{\ell}}\left(\frac{1}{2^{\ell-1}}-\frac{1}{10^{\ell-1}}\right) \geq A^{\ell}\left(\frac{1}{2b}\right)^{\ell}.\qquad\qquad \qedhere
\end{align*}

\end{proof}

\begin{claim}
	\label{claim:few-edges}
	There are at most $A^{\ell+1}(20b)^{\ell-1}C'\ell \log b$ edges in $G_{\ell}$  for $1\leq \ell\leq t$ where $C'$ is a constant.
	\end{claim}

\begin{proof}
We again proceed by induction on $\ell$.
Note that $|E(G_1)|=\binom{A}{2}$ and
\begin{align*}
     |E(G_{\ell})| \leq A |E(G_{\ell-1})| + \binom{A}{2} |E(H(A^{\ell-1}, k_{\ell-1}  ))| \leq A |E(G_{\ell-1})|+ C' A^{\ell+1} (10b)^{\ell-1}\ell \log b,
\end{align*}
since $G_{\ell}$ consists of $A$ copies of $G_{\ell-1}$ connected by $\binom{A}{2}$ copies of $H(A^{\ell-1},k_{\ell-1})$, each of which contains at most $C A^{\ell-1} (10b)^{\ell-1}\ell \log 10b \leq C' A^{\ell-1} (10b)^{\ell-1}\ell \log b$ edges.
Our inductive hypothesis yields the following bound on the number of edges in $G_{\ell}$:
\begin{align*}
   |E(G_{\ell})| &\leq  A |E(G_{\ell-1})|+ C' A^{\ell+1} (10b)^{\ell-1}\ell \log b \\ &\leq  A^{\ell+1}(20b)^{\ell-2}C'\ell \log b + C' A^{\ell+1} (10b)^{\ell-1}\ell \log b\leq C' A^{\ell+1} (20b)^{\ell-1}\ell \log b. \qedhere
\end{align*}
\end{proof}

Applying these two claims to $G_t$ implies that any $b$-coloring of $V(G_t)$ contains a monochromatic ordered path of length at least $$(2be^{t})^{t}\left(\frac{1}{2b}\right)^t= e^{t^2}=s,$$ 
while the number of edges in $G_t$ is at most
 $$A^{t+1}(20b)^{t-1}C't \log b \leq sA (40b^2)^t C' t \log b = se^{O(\log b \sqrt{\log s})} .$$
Since $G_t$ has the desired properties, this concludes the proof of the theorem.
\end{proof}
	
A result of this type can be used to obtain an upper bound on $\tilde{r}(P_r,P_s)$.
To do so, let $G$ be an ordered graph such that any $r$-coloring of the vertices contains a monochromatic path of length $s$.
Consider an arbitrary red-blue coloring of the edges of $G$.
If $G$ does not contain an ordered red copy of $P_r$, then we may define an $r$-coloring of the vertices of $G$ by assigning to each vertex $v$ the color $i$ which is the length of the longest ordered red path ending in $v$.
By assumption, this vertex coloring of $G$ contains a monochromatic ordered copy of $P_s$.
Since all edges in this path must be blue by definition of our vertex coloring, we have found a monochromatic ordered copy of $P_s$ in our original edge-coloring of $G$, which implies $\tilde{r}(P_r,P_s)\leq |E(G)|$.

In particular,  applying this argument with the ordered  graph $G_t$ defined in the proof of Theorem~\ref{thm:vertex-coloring} yields the upper bound $\tilde{r}(P_r,P_s)\leq se^{O( \log r \sqrt{\log s})}$.
While this is weaker than the result of Theorem~\ref{thm:size-ramsey},  this approach might be helpful to improve the upper bound on $\tilde{r}(P_r,P_s)$.

\section{Proof of Theorem~\ref{thm:multicolor}}


The proof of the upper bound in Theorem~\ref{thm:size-ramsey} extends quite easily to the $q$-color case, giving an ordered graph with $O(n^{2q-1} (\log n)^3)$ edges for which any $q$-coloring of the edges contains a monochromatic $P_n$. We start by increasing the number of vertices to $N=4n^q$ and creating a sequence of ordered graphs $G_0, G_1,\ldots,G_k$ as in Lemma~\ref{lemma:random-graph} with parameters $p_i=2^{-i}$, $\ell_i=2^{i+3}t$, and $m_i=2^{i+7}n^{q-1}t$. Then, we consider an arbitrary $q$-coloring of the edges of $G=\bigcup G_i$.

For each vertex $v$, we assign a vector $(c_1(v),\ldots, c_q(v))$ where $c_i(v)$ is the length of a longest ordered path in color $i$ ending at $v$.  Since each $c_i(v)\leq n-1$, there is a set $A\subseteq V(G)$ of size at least $\frac{N}{n^{q-1}}=4n$ such that $c_1(v),\ldots, c_{q-1}(v)$ are the same for all $v\in A$. All edges within $A$ have color $q$, so it suffices to find a long path in $G[A]$. (Note that the new value of $m_i$ is chosen to reflect the increased distance between points in $A$.) Finally, the inductive proof that we can find a copy of $P_n$ in color $q$ proceeds as in the proof of the upper bound in Theorem~\ref{thm:size-ramsey}.

The proof of the lower bound in Theorem~\ref{thm:multicolor} follows the same general structure as the proof of the lower bound in Theorem~\ref{thm:size-ramsey}: we start by identifying a subset of vertices of high degree, apply an Erd\H{o}s-Szekeres coloring to the edges in this set, and use a greedy vertex-coloring to define the edge-coloring on the rest of the graph.

In particular, set $d=\binom{\frac{n}{2}+q-1}{q-1}$ and fix some ordered graph $G$ with $d(n/2)^{q}$ edges. Let  $U_0$ be the set of vertices with at least $d$ neighbors preceding them in them in the ordering of $V(G)$. Then $U_0$ contains at most $(n/2)^q$ vertices,  so we can apply an Erd\H{o}s-Szekeres $q$-coloring on $E(G[U_0])$ to obtain an edge-coloring with no monochromatic $P_{n/2}$. Next, we greedily color the vertices in $V(G)\backslash U_0$ with $d$ colors and view each of these vertex colors $c$ as a $q$-tuple $(c_1,\ldots, c_q)$ with non-negative entries which sum to $n/2$. We color the edge between two vertices of colors $c$ and $c'$ with the smallest index $i$ such that $c_i<c'_i$. Since any coordinate can take on values ranging from 0 to $n/2$, we obtain a coloring of the edges outside of $G[U_0]$ in which the longest monochromatic ordered path is $P_{n/2}$. Therefore, this strategy yields an edge-coloring of $G$ which contains no monochromatic ordered $P_n$, as desired. 


\end{document}